\documentclass[a4paper,11pt]{amsart}
\usepackage[T1]{fontenc}
\usepackage[foot]{amsaddr}
\usepackage{geometry}
\usepackage{amsfonts}
\usepackage{mathrsfs}  
\usepackage{amsthm}
\usepackage{amssymb}
\usepackage{amsmath}
\usepackage{comment}
\usepackage{theoremref}
\usepackage{enumerate}
\usepackage{bbm}
\usepackage{bm}
\usepackage{comment}
\usepackage{mathtools}
\usepackage{a4wide}
\usepackage{xcolor}
\usepackage{lipsum}

\makeatletter
\newcommand{\proofpart}[2]{%
    \par
  \addvspace{\medskipamount}%
  \noindent\emph{Step #1: #2}\par\nobreak
  \addvspace{\smallskipamount}%
  \@afterheading
}
\makeatother

\DeclarePairedDelimiter\abs{\lvert}{\rvert}%
\DeclarePairedDelimiter\norm{\lVert}{\rVert}%

\makeatletter
\let\oldabs\abs
\def\abs{\@ifstar{\oldabs}{\oldabs*}}
\let\oldnorm\norm
\def\norm{\@ifstar{\oldnorm}{\oldnorm*}}
\makeatother

\makeatletter
\g@addto@macro\bfseries{\boldmath}
\makeatother

\newcommand{\C}{\mathbb{C}}

\newcommand{\T}{\mathbb{T}}

\newcommand{\Ka}{\mathcal{K}}
\newcommand{\conj}[1]{\overline{#1}}
\newcommand{\D}{\mathbb{D}}

\newcommand{\dist}[2]{\text{dist}( #1, #2 ) }

\newcommand{\m}{\textit{m}}

\renewcommand\Re{\operatorname{Re}}

\newcommand{\supp}[1]{\text{supp}({#1})}

\newtheorem{thm}{Theorem}[section]
\newtheorem{lemma}[thm]{Lemma}

\newtheorem{cor}[thm]{Corollary}
\newtheorem{prop}[thm]{Proposition}

\theoremstyle{definition}

\theoremstyle{definition}

\newcommand{\Addresses}{{
		\bigskip
		\footnotesize
		
		Adem Limani, \\ \textsc{Centre for Mathematical Sciences \\ Lund University \\
		Lund, Sweden}\\
		\texttt{adem.limani@math.lu.se}
		
	}}

\begin{document}
\title{\textbf{Summable analogous to the Ivashev-Musatov Theorems}} 

\author{Adem Limani} 
\address{Centre for Mathematical Sciences, Lund University, Sweden}
\email{adem.limani@math.lu.se}

\date{\today}

\begin{abstract}
We investigate summable analogues of the classical Ivashev-Musatov Theorem and threshold phenomenons alike. In the setting of weighted $\ell^1$ and Orlicz sequence spaces, we exhibit elements with critically pathological support and range, both topologically and in the sense of measure theory. Our results complement earlier works of J. P. Kahane, Y. Katznelson and T. W. K\"orner.




\end{abstract}

\maketitle
\section{Introduction}\label{SEC:INTRO}

\subsection{The classical Ivashev-Musatov Theorem} Let $\T$ denote the unit circle. A celebrated result in harmonic analysis, which traces back to the influential work of Ivashev-Musatov in the late 1950s (see \cite{ivavsev1958fourier}), goes as follows:

\begin{thm}[Ivashev-Musatov, 1958]\thlabel{THM:IM}
Let $(w_n)_n$ be positive real numbers satisfying the properties 
\begin{enumerate}
    \item[(i)] $\sum_n w_n^2 = + \infty$,
    \item[(ii)] there exists $C>1$ such that for any $n\geq1$:
    \[
    C^{-1} w_n \leq w_k \leq C w_n, \qquad n\leq k \leq 2n,
    \] 
    \item[(iii)] $w_n \downarrow 0$ or $n w_n \to \infty$.
\end{enumerate}
Then there exists a probability measure $\mu$ supported on a set of Lebesgue measure zero on $\T$, such that
\[
\abs{\widehat{\mu}(n)} \leq w_{|n|}, \qquad n=\pm 1, \pm 2, \dots 
\]
\end{thm}
The proof relies on beautiful ideas from harmonic analysis and, despite various efforts (see, for instance, \cite{korner1977theorem1}), the result is still regarded as difficult. However, its depth is indisputable, as it underlines that the cornerstone of Fourier analysis, namely Parseval's Theorem gives rise to a threshold phenomenon between localization, measured in terms of support, and smoothness, quantified by Fourier decay. As such, the Ivashev-Musatov Theorem (the IM-Theorem) naturally falls into the category of uncertainty principles in harmonic analysis. For a great survey on matters alike, see \cite{havinbook}. As indicated, the condition $(i)$ is absolutely necessary condition, while the conditions appearing in $(iii)$ were proved to be dispensable by T. W. K\"orner in \cite{korner1986theorem3}. This discussion naturally begs for the question whether $(ii)$ can also be dropped. Surprisingly, T. W. K\"orner (see Theorem 1.2 in \cite{korner1977theorem2}) showed that it cannot simply be removed.

\begin{thm}[K\"orner, 1978] \thlabel{THM:IMcond}There exists positive numbers $(w_n)_n$ with $w_n \downarrow 0$ that satisfies the hypothesis $(i)$ of Theorem 1.1, but such that any non-trivial distribution $S$ on $\T$ with 
\[
\abs{\widehat{S}(n)} \leq w_{|n|}, \qquad n=\pm 1, \pm 2, \dots .
\]
must have full support in $\T$.
    
\end{thm}
It was further indicated in \cite{korner1977theorem2} that condition (ii) may be slightly relaxed, but it remains unclear what an "optimal" condition should look like. This bears a resemblance to the Beurling-Malliavin multiplier theorem, where additional regularity conditions must be imposed on admissible majorants for functions with Fourier transforms of bounded bandwidth on the real line. For instance, see \cite{mashreghi2006beurling}, \cite{makarov2005meromorphic} and references therein. The following result was recently recorded as a by-product of the author's work in \cite{limani2024fourier}.

\begin{thm}[Theorem 2.7 in \cite{limani2024fourier}] \thlabel{THM:IMPhi} Let $\Phi$ be a continuous function which satisfies the assumption $\Phi(x)/x^2 \downarrow 0$ as $x \downarrow 0$. Then there exists a probability measure $\sigma$ with support on a set of Lebesgue measure zero in $\T$, such that 
\[
\sum_{n=-\infty}^\infty \Phi \left( \abs{\widehat{\sigma}(n) }\right) < \infty.
\]
\end{thm}
The measures constructed in \cite{limani2024fourier} are smooth in the sense of Zygmund, but we shall in Section \ref{SEC:SMALLSUPP}, demonstrate how such measures can easily be constructed via Riesz products. However, it is worth noting that, according to the main theorem of P. Duren in \cite{duren1965smoothness}, a result attributed to D. J. Newman and N. T. Varopoulos, classical Riesz products can never yield non-trivial smooth measures in the sense of Zygmund that are also singular with respect to the Lebesgue measure $dm$ on $\T$.

\subsection{A topological Ivashev-Musatov Theorem} 
Besides improving the classical IM-Theorem, T. W. K\"orner proved a beautiful topological version of it, which can be phrased as follows.

\begin{thm}[K\"orner's Topological IM-Theorem, 2003] \thlabel{THM:TOPIM}
    Let $(w_n)_n$ be positive real numbers satisfying the following conditions:
    \begin{enumerate}
        \item[(i)] $\sum_n w_n = + \infty$,
        \item[(ii)] there exists $C>1$ such that for any $n\geq1$:
    \[
    C^{-1} w_n \leq w_k \leq C w_n, \qquad n\leq k \leq 2n,
    \]
    \end{enumerate}
    Then there exists a positive essentially bounded function $f$ on $\T$, whose support has no interior, such that 
    \[
    \abs{\widehat{f}(n)} \leq w_{|n|}, \qquad n= \pm 1, \pm 2, \dots.
    \]
\end{thm}
We remark that the assumption $(i)$ is clearly necessary, as any function with absolutely summable Fourier coefficients is continuous on $\T$. However, it is not clear to what extent the additional assumption $(ii)$ can be relaxed. For further discussions, see \cite{korner2003topological}. T. W. K\"orner actually went a few steps further and asked whether the essentially bounded functions appearing in the statement of the Topological IM-Theorem are in fact typical. To this end, we denote by $\mathscr{C}$ be the collection of all non-empty compact subset of $\T$, equipped with the Hausdorff metric
\begin{equation}\label{EQ:HAUSMETRIC}
d_{\mathscr{C}} \left( E, F \right) := \sup_{\zeta \in F} \dist{\zeta}{E} + \sup_{\xi \in E} \dist{\xi}{F}.
\end{equation}
The following Baire version of the topological IM-Theorem was also obtained, see Theorem 15 in \cite{korner2003topological}.

\begin{thm}[K\"orner, 2003]\thlabel{THM:KÖRTHMBaire2} Let $(w_n)_n$ be positive real numbers satisfying the hypothesis $(i)-(ii)$ of \thref{THM:TOPIM}. Consider the collection $\mathscr{L}_w$ of ordered pairs $(f, E) \in L^\infty(\T,dm) \times \mathscr{C}$ satisfying the properties:
\begin{enumerate}
\item[(a.)] $0\leq f \leq 10$ $dm$-a.e on $\T$, 
\item[(b.)] $\supp{fdm} \subseteq E$,
\item[(c.)] $\lim_{|n|\to \infty} \abs{\widehat{f}(n)}w_{|n|}^{-1}=0$.

\end{enumerate}
Then the set $\mathscr{L}_w$ equipped with the metric
\[
d_w \left( (f, E), (g, K) \right) = d_{\mathscr{C}}(E,K) + \sup_{n} \abs{\widehat{f}(n)-\widehat{g}(n)}w_{|n|}^{-1} ,
\]
is complete, and the sub-collection of $(f,E)$ with $E$ containing an interior point forms a subset of the first category in $(\mathscr{L}_w, d_w)$.
\end{thm}

In the opposite, one may ask: \emph{how bad can an admissible majorant for Fourier coefficients of a continuous function possibly be?} This question was settled with precise accuracy by Kahane-Katznelson-de Leeuw in \cite{deleeuw1977surles}, where they proved that any $\ell^2$-sequence is pointwise majorized by the Fourier coefficients of a continuous function. 

\begin{thm}[Kahane-Katznelson-de Leeuw, 1977] For any sequence of positive real numbers $(w_n)_n$ satisfying 
\[
\sum_n w^2_n < \infty,
\]
there exists a continuous function $f$ on $\T$ such that 
\[
\abs{\widehat{f}(n)} \geq w_{|n|},  \qquad n=0,\pm 1, \pm 2, \dots
\]
\end{thm}
The proof of this remarkable Theorem is fairly technical, but not long, and relies on probabilistic constructions. A more streamlined proof is also outlined in the classical a book of Y. Katznelson, see the Appendix B of \cite{katznelson2004introduction}, which substantially differs from the original one, and from J. P. Kahane's proof in \cite{kahane1985some}.

\subsection{Summable analogues} 
The aforementioned results of T. W. K\"orner, Kahane-Katzneleson-de Leeuw are all about threshold phenomenons in the context of $\ell^\infty(1/w)$, the Banach space of distributions $S$ on $\T$, whose Fourier coefficients satisfy 
\[
\abs{\widehat{S}(n)} \leq C w_{|n|}, \, \qquad n=0,\pm 1, \pm 2, \dots.
\]
The primary aim of this paper is to investigate analogues of the aforementioned phenomena within the frameworks of weighted $\ell^1$ and Orlicz sequence spaces, which will be defined below. In this summable regime, the topological IM-type theorems and threshold phenomena of Katznelson type exhibit particularly intriguing features, as we shall demonstrate. The main results of the paper are presented in the following section

\section{Main results}
\subsection{Summable analogous to K\"orner's topological IM Theorem}



Our main intention in this subsection, is to prove summable analogues to K\"orner's topological IM-Theorem. To this end, let $\Phi$ be a convex function on $(0,\infty)$ with $\Phi(0+)=0$, and define the associated Orlicz sequence space $\ell^{\Phi}$ as the set of distributions $S$ on $\T$, whose Fourier coefficients satisfy
\[
\norm{S}_{\ell^{\Phi}} = \inf \left\{ M>0: \sum_n \Phi \left( \abs{\widehat{S}(n)}/M \right) \leq 1 \right\}< \infty.
\]
We shall further assume that $\Phi$ satisfies the so-called $\Delta_2$-condition: $\Phi(t/2) \asymp \Phi(t)$ as $t\to 0$ holds, and refer to such functions as \emph{Young functions}. One can show that $\ell^{\Phi}$ is a separable Banach space precisely when $\Phi$ is a Young function, and $S$ belongs to $\ell^{\Phi}$ if and only if
\[
\sum_n \Phi \left( \abs{\widehat{S}(n)} \right) < \infty.
\]
For further information on Orlicz sequence spaces, we refer the reader to the excellent survey by Lindenstrauss and Tzafriri in \cite{lindenstrauss2013classical}. Our first result is the following analogue of K\"orner's topological IM-Theorem, which allows us to exhibit an essentially bounded function with maximally bad support in topological sense.

\begin{thm}[Bad support in $\ell^{\Phi}$]\thlabel{THM:widebadsupp} Let $\Phi$ be a Young function with $\Phi(x)/x \downarrow 0$ as $x\downarrow 0$. Then there exists a positive essentially bounded function $f$ on $\T$ with $\supp{fdm}$ having no interior, such that 
    \[
    \sum_n \Phi \left( \abs{\widehat{f}(n)} \right) < \infty.
    \]
\end{thm}
We remark that our assumption is close to optimal, since if $\inf_{0<t<1} \Phi(t)/t >0$, then any distribution $S$ on $\T$ with $\Phi$-summable Fourier coefficients is actually continuous on $\T$. Next, we shall offer the analogous result in a slightly different framework. Given a sequence of positive numbers $w=(w_n)_{n=0}^\infty$, we denote by $\ell^1(w)$ the space of distributions $S$ on $\T$ satisfying 
\[
\norm{S}_{\ell^1 (w)} = \sum_{n} \abs{\widehat{S}(n)} w_{|n|} < \infty.
\]
In the context of the weighted $\ell^1$-spaces, our version of K\"orner's topological IM-Theorem goes as follows.

\begin{thm}[Bad support in $\ell^1(w)$]\thlabel{THM:widebadsuppw} Let $w_n \downarrow 0$. Then there exists a positive essentially bounded function $f$ on $\T$ with $\supp{fdm}$ having no interior, such that 
\[
\sum_n \abs{\widehat{f}(n)} w_{|n|} < \infty.
\]
\end{thm}
Again, this results is almost sharp, since if $\inf w_n >0$, then functions in $\ell^1(w)$ have absolutely summable Fourier coefficients, thus extend continuously up to $\T$. 

We now make two important observations regarding the proofs of our results. First, as we shall see in Subsection 3.4, \thref{THM:widebadsupp} and \thref{THM:widebadsuppw} cannot be obtained using methods based on sparse Fourier support, such as Sidon sets. Secondly, a sufficient condition for the indicator function of a set to have Fourier coefficients in $\ell^p$, for $1 < p < 2$, was established by D. J. Newman, see Theorem 7 in \cite{newman1964closure}. Roughly speaking, the condition involves a packing condition on the tail of complementary arcs, which can readily be verified for certain families of regular Cantor sets. However, this approach does not seem to readily extend, nor does it offer any simplification, in the more general setting considered here. 


Our proof relies on a technical construction of smooth, uniformly bounded functions with small amplitudes in the relevant spaces, combined with a Baire category argument. As such, the principal result of this subsection, which yields \thref{THM:widebadsupp} and \thref{THM:widebadsuppw} as immediate corollaries, takes the following form.

\begin{thm}\thlabel{THM:contX} Let $X$ be either $\ell^1(w)$ with $w_n \downarrow 0$, or $\ell^{\Phi}$ with $\Phi$ a Young function satisfying $\Phi(t)/t \downarrow 0$ as $t\downarrow 0$. Consider the collection $\mathscr{L}_{\mathscr{C}}(X) \subset L^\infty(\T,dm) \times \mathscr{C}$ of ordered pairs $(f,E)$ satisfying the properties:
\begin{enumerate}
    \item[(i)] $\supp{fdm}\subseteq E$,
    \item[(ii)] $0\leq f(\zeta) \leq 10$ $dm$-a.e $\zeta \in \T$,
    \item[(iii)] $\{\widehat{f}(n)\}_n \in X$,
\end{enumerate}
equipped with the metric 
\[
d_X \left( (f,E), (g,K) \right) := d_{\mathscr{C}}(E,K) + \norm{f-g}_X.
\]
Then the sub-collection of pairs $(f,E) \in \mathscr{L}_{\mathscr{C}}(X)$ with $E$ having empty interior is generic in the complete metric space $\left(\mathscr{L}_{\mathscr{C}}(X), d_X \right)$.
\end{thm}

\subsection{Functions with topologically wild range}
In the previous subsection, we exhibited summable analogues to K\"orner's topological IM Theorem on functions with maximally bad topological support. Here, we shall exhibit functions with maximally topologically bad ranges. To this end, we recall that a Lebesgue measurable function $f$ is unbounded in every neighborhood of $\T$, if for any arc $I\subseteq \T$ and any $M>0$:
\[
\m \left( \left\{\zeta \in I: \abs{f(\zeta)}>M \right\} \right) >0.
\]
Our main observation in this subsection goes as follows. 

\begin{thm}[Bad range]\thlabel{THM:BADRANGE} Let $X$ be either $\ell^1(w)$ with $w_n \downarrow 0$, or $\ell^{\Phi}$ with $\Phi$ a Young function satisfying $\Phi(t)/t \downarrow 0$ as $t\downarrow 0$. Then there exists a positive function $f \in L^2(\mathbb{T}, dm)$, unbounded in every neighborhood of $\mathbb{T}$, such that $(\widehat{f}(n))_n \in X$.
\end{thm}

  %





Note that since $f\in L^2(\T,dm)$, we always have that $f< \infty$ $dm$-a.e on $\T$. These results are arguably similar in spirit to \thref{THM:widebadsupp} and \thref{THM:widebadsuppw}, albeit technically less demanding, as they are susceptible to techniques involving functions with sparse Fourier support. In fact, we shall utilize a certain family of Riesz product, to exhibit functions which are locally unbounded at a single point, and the invoke a Baire category argument, to exhibit elements which are unbounded on a dense set. It may therefore not come as surprise that such elements will in fact also be generic. Furthermore, using the techniques of Riesz products, we shall in Section \ref{SEC:SMALLSUPP} outline a different proof of \thref{THM:IMPhi} and also prove an $\ell^1(w)$-analogue. The observant reader may object to why one cannot use a lacunary Hadamard series to produce a function with the desired properties. 
The advantage with lacunary series is that the corresponding examples are somewhat more explicit, while the drawback is that the method does not produce positive functions satisfying the required properties.

\subsection{Summable thresholds of Katznelson-type}
In the previous subsections, we found sharp conditions on when one can exhibited elements in $\ell^{\Phi}$ and $\ell^1(w)$, which are in, different ways, far from continuous. We now reverse the point of view and ask when can we exhibit continuous functions which barely belong to $\ell^{\Phi}$ and $\ell^1(w)$. Our principal result in this section goes as follows.

\begin{thm}\thlabel{THM:ContNOTl1w} Let $(w_n)_n$ be positive numbers satisfying the following properties:
\begin{enumerate}
    \item[(i)] $\sum_{n} w^2_n = +\infty$,
    \item[(ii)] $w_n$ is non-increasing or satisfies the doubling condition: there exists $C>1$ such that for all $n\geq1$:
    \[
    C^{-1} w_n \leq w_k \leq C w_n, \qquad n\leq k \leq 2n.
    \] 
    
\end{enumerate}
Then there exists a continuous function $f$ on $\T$ such that
\[
\sum_n \abs{\widehat{f}(n)}w_{|n|} = \infty.
\]
\end{thm}
Note that the result is essentially sharp, since if $\sum_n w_n^2< \infty$, Cauchy-Schwartz inequality readily implies that any $L^2(\T,dm)$-function is $\ell^1(w)$-summable. The proof of \thref{THM:ContNOTl1w} will involve a probabilistic argument, in a similar flavor to Katznelson's proof of the Kahane-Katznelson-de Leeuw Theorem. Reformulating \thref{THM:ContNOTl1w} by means of duality, we obtain the following complementary result to the classical IM-Theorem.

\begin{cor}\thlabel{COR:FOURSTIL} Let $(w_n)_n$ be positive numbers satisfying the hypothesis $(i)-(ii)$ of \thref{THM:ContNOTl1w}. Then there exists a real-valued sequence $(a_n)_n$ with 
\[
\abs{a_n} \leq w_{|n|}, \qquad n=0,\pm 1, \pm 2, \dots,
\]
which is not the Fourier coefficients of a complex finite Borel measure on $\T$.
\end{cor}

In the context of Orlicz sequence spaces, we have the following more classical result, which are essentially due to Y. Katznelson.

\begin{thm}[Katznelson] \thlabel{THM:KATZ1} For any continuous function $\Psi$ with $\Psi(t)/t^2 \uparrow +\infty$ as $t\downarrow 0$, there exists a continuous function $f$ on $\T$ such that 
\[
\sum_n \Psi \left( \abs{\widehat{f}(n)} \right) = +\infty.
\]
\end{thm}

We also record the the following dual reformulation of \thref{THM:KATZ1}

\begin{cor}[Katznelson] \thlabel{THM:KATZ2} For any continuous function $\Phi$ with $\Phi(t)/t^2 \downarrow 0$ as $t\downarrow 0$, there exists positive numbers $(a_n)_n$ with
\[
\sum_n \Phi \left( a_n \right) < +\infty,
\]
but $(a_j)_j$ are not the Fourier coefficients of a complex finite Borel measure on $\T$.
\end{cor}
For the case $\Psi(t)=t^{2-\varepsilon}$ and $\Phi(t)= t^{2+\varepsilon}$, with $\varepsilon>0$ small, outlines of the relevant results can be found in Chapter IV of \cite{katznelson2004introduction}. In Section 6, however, we briefly demonstrate how further adjustments of these ideas extend to the finer scales of Orlicz spaces. Notably, the proof of \thref{THM:KATZ1} can be deduced from the Kahane-Katznelson-de Leeuw Theorem, but it is far from necessary in order to establish \thref{THM:KATZ1}. We shall in Section \ref{SEC:THREKATZ} explain why \thref{COR:FOURSTIL} and \thref{THM:KATZ2} may be interpreted as results on elements with maximally bad measure theoretical range.

\subsection{Organization and notation}
The paper is organized as follows: In Section \ref{SEC:LBS} we prove our summable analogues of K\"orner's topological IM-Theorem. Our proof will involve the construction of smooth functions which vanish on neighborhood of a prescribed point in $\T$, are uniformly bounded, and with small amplitudes. This construction will be combined with a softer Baire category argument. Section \ref{SEC:LSBR} is devoted to exhibiting positive elements with maximally bad range, and is based on a certain family of Riesz products, in conjunction with a different Baire category argument. In Section \ref{SEC:SMALLSUPP}, we provide two additional constructions based on Riesz products, were distinct proofs of \thref{THM:IMPhi} and sharp $\ell^1(w)$-analogue. Our final Section \ref{SEC:THREKATZ} contains the proof of our results announced in subsection 2.3, with principal effort devoted to \thref{THM:ContNOTl1w}.

For two positive numbers $A, B >0$, we will frequently use the notation $A \lesssim B$ to mean that $A \leq cB$ for some positive constant $c>0$. If both $A\lesssim B$ and $B \lesssim A$ hold, we will write $A\asymp B$. Additionally, absolute constants will generally be denoted by $C$, even though the value of $C$ may vary from line to line.

\subsection{Acknowledgements} This research was supported by a stipend from the 
Knut \& Alice Wallenberg Foundation (grant no. 2021.0294). The author would like to thank Eskil Rydhe and Artur Nicolau for valuable discussions during the preparation of this manuscript.

\section{Functions with exceptionally bad support}\label{SEC:LBS}


\subsection{Smooth localizing functions}

The main focus of this subsection is to collect lemmas concerning smooth functions that vanish in a neighborhood of a prescribed set, and approximate the constant function $1$ in the norms of $\ell^\Phi$ and $\ell^1(w)$, respectively. Our first step towards this end, requires a slight modification of Lemma 20 in \cite{korner2003topological} by T. W. K\"orner. Roughly speaking, this is a sampling type argument. 

\begin{lemma}\thlabel{LEM:KOR} Let $0<\eta<1/2$ and $N>0$ be an integer. There exists constants $c_1(N), c_2(N) >0$, independent on $\eta$‚ and a smooth function $\chi_{\eta,N}$ on $\T$, which satisfies the following properties:
\begin{enumerate}
    \item[(i)] $\chi_{\eta,N}(\zeta)=1$ on an arc centered at $\zeta=1$ of length $c_1(N) \eta$.
    \item[(ii)] $-1/N \leq \chi_{\eta,N}(t)\leq 1$ on $\T$,
    \item[(iii)] $\int_{\T}\chi_{\eta, M} dm=0$,
    \item[(iv)]
    \[
    \abs{\widehat{\chi_{\eta,N}}(n)} \leq \min \left( \eta, c_2(N) \eta^{-1} \abs{n}^{-2} \right), \qquad n\neq 0.
    \]
\end{enumerate}
\end{lemma}
\begin{proof} Let $I_{\delta}\subset \T$ denote the closed arc centered at $\zeta =1$ of length $\delta$. Fix a smooth function $\varphi$ on $\T$ with the properties
\begin{equation}\label{EQ:varphi}
\supp{\varphi}\subseteq I_{1}, \qquad 0\leq \varphi(\zeta) \leq 1, \qquad \varphi(\zeta)= 1, \qquad \zeta \in I_{1/4}
\end{equation}

We now scale and then localize $\varphi$ by the parameter $\eta\in (0,1)$, thus we consider the smooth function $\varphi_{\eta}(\zeta)= 1_{I_\eta}(\zeta) \varphi(\zeta^{1/\eta})$ for $\zeta \in \T$, where $1_{I_\eta}$ denotes the indicator function wrt $I_\eta$. A simple change of variable followed by integrating by parts twice, easily gives the Fourier decay 
\[
\abs{\widehat{\varphi_\eta}(n)} \leq \min \left( \eta, c_1 \eta^{-1} |n|^{-2} \right), \qquad n=0,\pm 1, \pm 2, \dots,
\]
where the constant $c_1>0$ only depends on the fixed function $\varphi$. Fix a large integer $N>0$ and consider $N$ uniformly distributed points $\zeta_{k,\eta} = e^{i\eta(k+N)/2N}$ for $k=1,2, \dots, N$ on one component of $I_{2\eta} \setminus I_{\eta}$, and define the associated measure
\[
\nu_{\eta,N} := \frac{1}{N}\sum_{k=1}^{N}\left( \delta_1 - \delta_{\zeta_{k,\eta}} \right),
\]
where $\delta_\zeta$ denotes the Dirac measure supported at $\zeta \in \T$. Based on the simple observation $\abs{\widehat{\delta_1 - \delta_{e^{it}} }(n)}= 2\abs{\sin(\pi nt)}$, it is straightforward to verify that 
\begin{equation}\label{EQ:nu}
\widehat{\nu_{\eta,N}}(0)=0, \qquad \abs{\widehat{\nu_{\eta,N}}(n)} \leq 2\min\left( 1, 2\pi \eta |n| \right), \qquad n\neq 0.
\end{equation}
Now consider the functions defined in terms of the convolution
\[
\chi_{\eta,N}(\zeta) := \varphi_{\eta /16N} \ast \nu_{\eta,N}(\zeta) = \varphi_{\eta /16N}(\zeta) - \frac{1}{N} \sum_{k=1}^N \varphi_{\eta /16N}(\zeta \cdot \conj{\zeta_{k,\eta}}), \qquad \zeta \in \T. 
\]
For instance, taking $c_1(N)= 2^{-16}N^{-1}$, we ensure that $\zeta \conj{\zeta_{k,\eta}} \notin I_{\eta/16N}$ whenever $\eta \in I_{c_1(N)\eta}$, thus $(i)$ holds. Property $(ii)$ is immediate from the fact that $0\leq \varphi \leq 1$, while the properties $(iii)-(iv)$ follow from \eqref{EQ:varphi} and \eqref{EQ:nu}, together with the convolution formula
\[
\widehat{\chi_{\eta,N}}(n) = \widehat{\varphi_\eta}(n) \widehat{\nu_{\eta,N}}(n).
\]
This completes the proof.
\end{proof}

With this lemma at hand, we now turn to our first key observation in this subsection, phrased in the framework of Orlicz spaces.

 \begin{lemma}\thlabel{LEM:MAINKÖR} Let $\Phi$ be Young function with $\Phi(t)/t \downarrow 0$ as $t\to 0$. Then for any $\varepsilon>0$ and any $a \in \T$, there exists $\psi_{\varepsilon} \in C^\infty(\T)$ such that 
 \begin{enumerate}
     \item[(i)] $0\leq\psi_{\varepsilon}(\zeta) \leq 1+\varepsilon$, for $\zeta \in \T$,
     \item[(ii)] $\psi_{\varepsilon}(\zeta)=0$ in a neighborhood of $a\in \T$, whose length tends to $0$ as $\varepsilon\to 0+$,
     \item[(iii)] $\int_{\T} \psi_{\varepsilon}dm=1$,
     \item[(iv)] $\sum_{n\neq 0} \Phi \left( \abs{\widehat{\psi}_{\varepsilon}(n)} \right) \leq \varepsilon.$
 \end{enumerate}
 
 \end{lemma}

\begin{proof}
Fix $\varepsilon>0$, and pick a large integer $N>1/\varepsilon$ and let $0<\eta<1/2$ to determined later. Utilizing K\"orner's \thref{LEM:KOR}, we can exhibit an element $\chi_{\eta, N}$ satisfying the properties $(i)-(iv)$ therein. Now the idea is to consider functions of the form
\[
\psi_\eta (\zeta) = 1- \chi_{\eta, N}(\zeta \cdot \conj{a}).
\]
Form property $(i)$ and $(ii)$ of \thref{LEM:KOR}, in conjunction with our choice of $\gamma, N>0$ in relation to $\varepsilon$, it follows that that 
\[
\psi_\eta (\zeta) = 0, \qquad \zeta \in I_{c_1(N)\eta}(a),
\]
and that $0\leq \psi_\eta \leq 1+ 1/N \leq 1+\varepsilon$ on $\T$. Hence we have verified that $(i)-(ii)$ holds for $\psi_{\eta}$. Clearly $(iii)$ also holds, since $\chi_{\eta,N}$ have zero average. At last, we now estimate the non-zero Fourier coefficients of $\psi_{\eta}$. To this end, it suffices to only to show that $\chi_{\eta,N}$ satisfies the required estimate, since we can simply apply translations. Now utilizing the monotonicity of $\Phi$, the assumption that $\Phi(t)/t \downarrow 0$ as $t\to 0$, and the Fourier estimate of $\chi_{\eta,N}$, we get
\begin{multline*}
\sum_{n\neq 0} \Phi \left( \abs{\widehat{\chi_{\eta,N}}(n)}\right) \leq \sum_{0<|n|\leq 1/\eta} \Phi \left( \eta \right) + 
\sum_{|n| > 1/\eta} \Phi \left( \abs{\widehat{\chi_{\eta,N}}(n)}\right) \\
\leq \frac{2\Phi(\eta)}{\eta} + \frac{\Phi(\eta)}{\eta}\sum_{|n|>1/\eta} \abs{\widehat{\chi_{\eta,N}}(n)} \leq \frac{2\Phi(\eta)}{\eta} + \frac{\Phi(\eta)}{\eta^2} c_2(N) \sum_{|n|>1/\eta} \frac{1}{n^2} \leq (2 + 4c_2(N)) \frac{\Phi(\eta)}{\eta}.
\end{multline*}
Now since $\Phi(\eta)/\eta \downarrow 0$ as $\eta\to 0$, we can choose $\eta $ sufficiently small, such that $(iv)$ is fulfilled. This completes the proof of our lemma.

\end{proof}

We now exhibit the weighted $\ell^1$ version, which looks slightly different.

\begin{lemma}\thlabel{LEM:KÖRl1w} Let $w_n \downarrow 0$. Then for any $\varepsilon>0$ and any $a \in \T$, there exists $\psi_{\varepsilon} \in C^\infty(\T)$ such that 
 \begin{enumerate}
     \item[(i)] $0\leq\psi_{\varepsilon}(\zeta) \leq 1+\varepsilon$, for $\zeta \in \T$,
     \item[(ii)] $\psi_{\varepsilon}(\zeta)=0$ in a neighborhood of $a\in \T$, whose length tends to $0$ as $\varepsilon\to 0+$,
     \item[(iii)] $\int_{\T} \psi_{\varepsilon}dm=1$,
     \item[(iv)] $\sum_{n\neq 0}  \abs{\widehat{\psi}_{\varepsilon}(n)} w_{|n|} \leq \varepsilon$,
     \item[(v)] $\abs{\widehat{\psi_\varepsilon}(m)}\leq \min \left(\varepsilon, 10\varepsilon^{-1} \abs{m}^{-2} \right)$ for all $m\neq0$.
 \end{enumerate}
\end{lemma}

\begin{proof} Again fix $\varepsilon>0$ and pick $N> 1/\varepsilon$ and let $0<\eta<1/2$ to be chosen later. Again, we apply \thref{LEM:KOR} with $\chi_{\eta, N}$ and consider functions of the form
\[
\psi_\eta (\zeta) = 1- \chi_{\eta, N}(\zeta \cdot \conj{a}).
\]
As in the proof of \thref{LEM:MAINKÖR}, we readily see that $(i)-(iii)$ and $(v)$ holds, hence it remains only to verify $(iv)$. To this end, utilizing the monotonicity of $(w_n)_n$, we get
\begin{multline*}
\sum_{n\neq 0} \abs{\widehat{\chi_{\eta, N}}(n)} w_{|n|} \leq 2\eta \sum_{0<n\leq 1/\eta} w_n + \frac{2c_2(N)}{\eta} \sum_{n>1/\eta} \frac{w_n}{n^2} \leq 2\eta \sum_{0<n\leq 1/\eta} w_n + 4c_2(N) w_{[1/ \eta]},
\end{multline*}
where $[1/\eta]$ denotes the integer part of $1/\eta$. Now since $w_n \to 0$, its Cesar\'o mean also satisfies
\[
\lim_{\eta \to 0+} \eta \sum_{0<n\leq 1/\eta} w_n = 0.
\]
Therefore, choosing $\eta$ sufficiently small, we can also ensure that $(iv)$ holds.
\end{proof}

\subsection{Functional analytic setting}
In this subsection, we will always assume that $\Phi$ is a Young function and that $w_n \downarrow 0$. We observe that $\ell^1(w)$ becomes a separable Banach space that contains the trigonometric polynomials as a dense subset. Also note that the dual space of $\ell^1(w)$ can be identified with $\ell^\infty(1/w)$ in the densely-defined pairing:
\[
\abs{\sum_n \widehat{T}(n) \conj{\widehat{g}(n)}} \leq \norm{T}_{\ell^1(w)} \sup_{n} \frac{\abs{\widehat{g}(n)}}{w_{|n|}},
\]
where $T$ is a trigonometric polynomial. In order to introduce duality in the Orlicz setting, we need the notion of the Legendre transform. Given a Young function $\Phi$, we define its Legendre transform
\[
\Phi^*(x) := \sup_{y\geq 0} \left( xy - \Phi(y) \right), \qquad x\geq 0,
\]
which again is a convex non-decreasing function with $\Phi^*(0)=0$. It is a classical fact that the Legendre transform implies the so-called H\"older-type inequality 
\[
\abs{\sum_n \widehat{f}(n) \conj{\widehat{g}(n)}} \leq 2 \norm{f}_{\ell^\Phi} \norm{g}_{\ell^{\Phi^*}},
\]
which at its turn gives rise the dual relation $\left( \ell^{\Phi} \right)' \cong \ell^{\Phi^*}$. We remark that the Legendre transform $\Phi^*$ of a Young function $\Phi$ need itself not be a Young function, but it is so if and only if 
\[
\liminf_{t\to 0+} \frac{t\Phi'(t)}{\Phi(t)} > 1.
\]
In our framework, this assumption will implicitly hold, since we shall exclusively work under the assumption that $\Phi(t)/t \downarrow 0$ as $t\to 0$. Therefore, there is no great loss in assuming that both $\ell^{\Phi}$ and its dual $\ell^{\Phi^*}$ are separable, thus reflexive Banach spaces. For further details on these matters, we refer the reader to  Lindenstrauss and Tzafriri in \cite{lindenstrauss2013classical}. 

We shall now record a couple of useful lemmas summarizing our framework. The first one is the following natural extension of the classical Young inequality for convolutions, due to O'Neil, formulated in the context of Orlicz sequence spaces.

\begin{thm}[O'Neil, \cite{o1965fractional}] \thlabel{THM:Oneil} For any Young function $\Phi$, we have 
\[
\norm{f\ast g}_{\ell^{\Phi}} \leq \norm{f}_{\ell^1} \norm{g}_{\ell^{\Phi}}.
\]
\end{thm}

Throughout this section, we shall denote by $X$ either the space $\ell^{\Phi}$ or $\ell^1(w)$, where $\Phi$ is a Young function and $w_n \downarrow 0$. We now record the following simple lemma, which will be useful for our further developments.

\begin{lemma} \thlabel{LEM:SPhi} The collection $\mathscr{S}_{X}\subset X$ of elements $f\in L^\infty(\T,dm)$ satisfying the properties
\begin{enumerate}
    \item[(i)] $0\leq f(\zeta) \leq 10$, $dm$-a.e $\zeta \in \T$,
    \item[(ii)] $(\widehat{f}(n))_n$ belongs to $X$,
\end{enumerate}
forms a closed subset of $X$.
\end{lemma}
\begin{proof}
Since the set is convex, it suffices to show that it is weakly closed. Now if $f_j \in \mathscr{S}_{X}$ with $f_j \to f$ weakly in $X$, then for any positive smooth function on $\varphi$ on $\T$, one has 
\[
0 \leq \sum_n \widehat{f_j}(n) \widehat{\varphi}(n) = \int_{\T} f_j \varphi dm \leq 10 \int_{\T} \varphi dm, \qquad \forall j.
\]
Sending $j \to \infty$, we see that the same holds for $f$. For instance, taking $\varphi$ to be the Poisson kernel wrt to the unit-disc $\{|z|<1\}$ and using standard properties of Poisson kernels, we conclude that $f$ satisfies $(i)-(ii)$.
\end{proof}
Note that $\mathscr{S}_{\Phi}$ and $\mathscr{S}_w$ are only cones in $\ell^{\Phi}$ and $\ell^1(w)$, respectively. Now consider the collection $\mathscr{L}_{X} \subset \mathscr{S}_{X} \times \mathscr{C}$ of ordered pairs $(f,E)$ with $f\in \mathscr{S}_{X}$ and $E\subseteq \T$ compact, such that
\[
\supp{f} \subseteq E.
\]
At last, we now record the following simple lemma, whose proof is almost immediate from \thref{LEM:SPhi} and standard properties of support.

\begin{lemma}\thlabel{LEM:LPHIcomplete} The set $\mathscr{L}_{X}$ of ordered pair $(f,E)$ equipped with the metric 
\[
d_{X}\left( (f,E), (g,K) \right) := \norm{f-g}_{X} + d_{\mathscr{C}}(E,K),
\]
becomes a complete metric space.
\end{lemma}

\subsection{A Baire category argument}\label{SSEC:Baire}
Here we outline an approach involving Baire category approach, which is largely inspired by the work of T. K\"orner \cite{korner2003topological}. The application of Baire arguments of similar kinds have previously appeared in thew work of J. P. Kahane in \cite{kahane2000baire}, and R. Kaufman in \cite{kaufman1967functional}. With the preparatory results from the previous subsections at our disposal, we now turn to the principal observation in this section.

\begin{prop}\thlabel{PROP:MAINPROPPHI} Let $\Phi$ be a Young function with $\Phi(t)/t \downarrow 0$ as $t \to 0$. For any $a\in \T$, consider the set 
\[
\mathscr{E}_a = \left\{(f,E) \in \mathscr{L}_{\Phi}: E \, \text{does not meet an open arc containing } \, a \right\}.
\]
Then $\mathscr{E}_a$ is open and dense in the metric space of $(\mathscr{L}_{\Phi},d_{\Phi})$.
\end{prop}
\begin{proof}
To see why $\mathscr{E}_a$ open, let $(f,E) \in \mathscr{E}_a$ and pick $\delta>0$ such that the arc $I_{2\delta}(a)$ centered at $a$ of length $4\delta$ does not meet $E$. Now if $d_{\Phi}\left( (f,E), (g,K) \right) \leq \delta/2$, then we in particular have that $d(E,K) \leq \delta /2$, hence we can infer that $I_{\delta/2}(a) \cap K = \emptyset$. This shows that $\mathscr{E}_a$ is indeed open.
    

In order to verify that $\mathscr{E}_a$, it suffices to show that for any $(f,E) \in \mathscr{L}_{\Phi}$ and any fixed $\delta>0$, there exists $(g,K) \in \mathscr{E}_a$, such that 
\[
d_{\Phi} \left( (f,E), (g,K) \right) < \delta.
\]
To this end, we fix $(f,E) \in \mathscr{L}_{\Phi}$ and $\delta>0$. Now fix $\varepsilon>0$ and consider a $C^\infty$-smooth kernel $k_{\varepsilon}$ with the properties:
\begin{enumerate}
    \item[(i)] $k_{\varepsilon}\geq 0$,
    \item[(ii)] $\int_{\T} k_{\varepsilon} dm=1-\varepsilon$,
    \item[(iii)] $\supp{k_{\varepsilon}}\subset I_{\varepsilon}$, where $I_{\varepsilon}$ is the arc centered at $\zeta=a$ of length $\varepsilon$.
\end{enumerate}
Now by means of considering $f_{\varepsilon}:= f \ast k_{\varepsilon}$, it follows that $\supp{f_{\varepsilon}} \subseteq K + I_{\varepsilon}$, thus $d(\supp{f_{\varepsilon}}, K) \leq \varepsilon$. We also have that
\[
0\leq f_\varepsilon(\zeta) \leq \norm{f}_\infty \int_{\T} k_{\varepsilon} dm\leq 10(1-\varepsilon).
\]
Furthermore, since
\[
\abs{\widehat{f_{\varepsilon}}(n)-\widehat{f}(n)} =\abs{\widehat{f}(n)}\cdot \abs{\widehat{k_\varepsilon}(n)-1} \leq 2\abs{\widehat{f}(n)} , \qquad n\in \mathbb{Z},
\]
we have for any $\varepsilon>0$ and any $N>0$ that
\[
\sum_n \Phi \left( \abs{\widehat{f_{\varepsilon}}(n)-\widehat{f}(n)} \right) \lesssim \sum_{\abs{n}\leq N} \Phi \left( \abs{\widehat{f}(n)}\cdot \abs{\widehat{k_\varepsilon}(n)-1} \right) + \sum_{|n|>N} \Phi \left( \abs{\widehat{f}(n)} \right).
\]
Sending $\varepsilon \to 0+$ first, then $N \to \infty$, we conclude that $f_{\varepsilon} \to f$ in $\ell^{\Phi}$ and $f_{\varepsilon} \in \ell^{\Phi}$. Here, we used the assumption $\Phi(t/2) \asymp \Phi(t)$ to deduce convergence in $\ell^{\Phi}$ from the estimates in the previous display. Therefore, by means of passing to the such a pair $(f_{\varepsilon}, \supp{f_{\varepsilon}})$, we may assume that $f \in C^\infty(\T)$ with $\norm{f}_{\infty}\leq 1-\delta/10$. Now according to \thref{LEM:MAINKÖR}, there exists there exists a $\psi_{\varepsilon} \in C^\infty(\T)$ such that 
 \begin{enumerate}
     \item[(i)] $0<\psi_{\varepsilon}(\zeta) \leq 1+\varepsilon$, for $\zeta \in \T$,
     \item[(ii)] $\psi_{\varepsilon}(\zeta)=0$ on an open arc $A_{\varepsilon}$ centered at $a\in \T$ of length at most, say $10\varepsilon$,
     \item[(iii)] $\int_{\T} \psi_{\varepsilon}dm=1$,
     \item[(iv)] $\sum_{n\neq 0} \Phi \left( \abs{\widehat{\psi}_{\varepsilon}(n)} \right) \leq \varepsilon.$
 \end{enumerate}
We now form the compact subset $K_{\varepsilon}:=E \setminus A_{\varepsilon}$ and consider the $C^{\infty}(\T)$-function $g_{\varepsilon}= f \cdot \psi_{\varepsilon}$. Property $(ii)$ of $\psi_{\varepsilon}$ implies that $\supp{g_\varepsilon}\subseteq K_{\varepsilon}$, thus $d(E,K_{\varepsilon}) \leq 10\varepsilon$. Furthermore, property $(i)$ in conjunction with the supremum-norm assumption on $f$ gives 
\[
\norm{g_{\varepsilon}}_{\infty} \leq (1+\varepsilon)(1-\delta/10) <1
\]
if $0<\varepsilon<\delta/10$ sufficiently small. This ensures that $(g_\varepsilon , K_\varepsilon) \in \mathscr{E}_a$. Now, since 
\[
g_{\varepsilon} -f = f \left(\psi_\varepsilon - \int_{\T}\psi_\varepsilon dm  \right),
\]
an application of \thref{THM:Oneil} and the properties of $\psi_\varepsilon$ gives  
\[
\norm{g_{\varepsilon}-f}_{\ell^{\Phi}} \leq \norm{\psi_\varepsilon - \int_{\T}\psi_\varepsilon dm }_{\ell^{\Phi}} \norm{f}_{\ell^1} \leq C\varepsilon \norm{f}_{\ell^1}.
\]
Choosing $\varepsilon>0$ small enough, depending on $f$ and $\delta>0$, we can ensure that 
\[
d_{\Phi} \left( (f,E) , (g_\varepsilon, K_\varepsilon ) \right) \leq \delta,
\]
hence $\mathscr{E}_a$ is dense in $\mathscr{L}_{\Phi}$.

\end{proof}

We now turn our attention to the weighted $\ell^1$-analogue of \thref{PROP:MAINPROPPHI}, whose proof is principally similar, but requires a bit more work.

\begin{prop}\thlabel{PROP:MAINPROPl1w} Let $w_n \downarrow 0$. For any $a\in \T$ consider the set of ordered pairs
\[
\mathscr{E}_a = \left\{ (f,E) \in \mathscr{L}_w: E \, \text{does not meet an open arc containing} \, a \right\}.
\]
Then $\mathscr{E}_a$ is a open and dense set in the metric space $(\mathscr{L}_w, d_w)$.
\end{prop}
\begin{proof}
The proof that $\mathscr{E}_a$ is open is simple and very similar to the proof of \thref{PROP:MAINPROPPHI}, hence we omit the details. In order to verify that $\mathscr{E}_a$ is dense, it suffices to show that for any $(f,E) \in \mathscr{L}_w$ and any $\delta>0$, there exists $(g,K) \in \mathscr{E}_a$ such that 
\[
d_w \left( (f,E), (g,K) \right) < \delta.
\]
Again, a simple argument involving taking convolution with a smooth kernel, we may actually assume that $f \in C^\infty(\T)$. According to \thref{LEM:KÖRl1w}, there exists $\psi_{\varepsilon} \in C^\infty(\T)$ such that 
 \begin{enumerate}
     \item[(i)] $0\leq\psi_{\varepsilon}(\zeta) \leq 1+\varepsilon$, for $\zeta \in \T$,
     \item[(ii)] $\psi_{\varepsilon}(\zeta)=0$ in a neighborhood of $a\in \T$, whose length tends to $0$ as $\varepsilon\to 0+$,
     \item[(iii)] $\int_{\T} \psi_{\varepsilon}dm=1$,
     \item[(iv)] $\sum_{n\neq 0}  \abs{\widehat{\psi}_{\varepsilon}(n)} w_{|n|} \leq \varepsilon$,
     \item[(v)] $\abs{\widehat{\psi_\varepsilon}(m)}\leq \min \left(\varepsilon, 10\varepsilon^{-1} \abs{m}^{-2} \right)$ for all $m\neq0$.
 \end{enumerate}
Now set $f_\varepsilon = f \cdot \psi_\varepsilon$ and $E_\varepsilon = E \setminus I_{\varepsilon}(a)$. As before, it easily follows that the pair $(f_\varepsilon, E_\varepsilon)$ belongs to $\mathscr{L}_w$ and we have $d(E,E_\varepsilon)\leq \varepsilon$. Therefore, it only remains to show that 
\[
\sum_{n} \abs{\widehat{f_\varepsilon}(n)-\widehat{f}(n)}w_{|n|} \to 0, \qquad \varepsilon \to 0+.
\]
Since all functions involved are real-valued, we only need to estimate the sum for $n\geq 0$. Note that 
\[
\widehat{f_\varepsilon}(n)-\widehat{f}(n) = \sum_{m\neq n} \widehat{f}(m) \widehat{\psi_\varepsilon}(n-m) = \left(\sum_{|m|\leq n/2} + \sum_{\substack{|m|>n/2 \\ m\neq n}} \right) \widehat{f}(m) \widehat{\psi_\varepsilon}(n-m).
\]
Using the property $(v)$ and the assumption that $f$ is smooth on $\T$, we get
\[
\abs{\sum_{\substack{|m|>n/2 \\ m\neq n}} \widehat{f}(m) \widehat{\psi_\varepsilon}(n-m)} \leq \varepsilon \sum_{|m|>n/2} \abs{\widehat{f}(m)} \leq   \frac{C\varepsilon}{(1+n)^{10}}
\]
Therefore, we obtain
\[
\sum_{n\geq 0} \abs{\sum_{\substack{|m|>n/2 \\ m\neq n}} \widehat{f}(m) \widehat{\psi_\varepsilon}(n-m)} w_{n} \leq C \varepsilon \sum_{n\geq 0} \frac{1}{(1+n)^{10}} \leq C' \varepsilon.
\]
It now only remains to estimate the sum
\[
\sum_{n\geq 0} \sum_{|m|\leq n/2} \abs{\widehat{f}(m) \widehat{\psi_\varepsilon}(n-m)} w_{n}.
\]
To this end, we shall split the above sum into two parts. For sufficiently large $n$, we use property $(v)$ and monotonicity of $(w_n)_n$, which yields
\begin{multline*}
\sum_{n\geq 1/\varepsilon} \sum_{|m|\leq n/2} \abs{\widehat{f}(m) \widehat{\psi_\varepsilon}(n-m)} w_{n} \leq \frac{10}{\varepsilon }\sum_{n\geq 1/\varepsilon} \sum_{|m|\leq n/2} \frac{\abs{\widehat{f}(m)}}{|n-m|^2} w_n \\ \leq \frac{C}{\varepsilon }\sum_{n\geq 1/\varepsilon} \frac{w_n}{n^2} \sum_{|m|\leq n/2} \abs{\widehat{f}(m)} \leq
\frac{C\norm{f}_{\ell^1}}{\varepsilon} w_{\left[1/\varepsilon\right]} \sum_{n\geq 1/\varepsilon} \frac{1}{n^2} \leq C' w_{\left[1/\varepsilon \right]}.
\end{multline*}
On the hand, when $n$ is not too large, we again use the monotonicity of $(w_n)_n$ but now assumption $(iv)$, in order to deduce
\begin{multline*}
\sum_{0\leq n <1/\varepsilon} \sum_{|m|\leq n/2} \abs{\widehat{f}(m) \widehat{\psi_\varepsilon}(n-m)} w_{n} \leq 
\sum_{0\leq n < 1/\varepsilon} \sum_{|m|\leq n/2} \abs{\widehat{f}(m)}\abs{ \widehat{\psi_\varepsilon}(n-m)} w_{|n-m|} = \\
\sum_m \abs{\widehat{f}(m)} \sum_{2|m|\leq n \leq 1/\varepsilon}
\abs{ \widehat{\psi_\varepsilon}(n-m)} w_{|n-m|} \leq  \norm{f}_{\ell^1}\varepsilon.
\end{multline*}
Combining all terms, we arrive at the estimate
\[
\sum_n \abs{\widehat{f_\varepsilon}(n)-\widehat{f}(n)}w_{|n|} \leq C_f \max( \varepsilon, w_{[1/\varepsilon]}),
\]
where $C_f>0$ is a constant only depending on $f$. Letting $\varepsilon \downarrow 0$, finishes the proof.
\end{proof}

With this at hand, we can now easily complete the proof our main result in Subsection 2.1

\begin{proof}[Proof of \thref{THM:contX}] Again, let $X= \ell^{\Phi}$ or $X= \ell^1(w)$. Pick a countable dense subset $\{a_j\}_j \subset \T$ and note that according to \thref{PROP:MAINPROPPHI} and \thref{PROP:MAINPROPl1w}, and the Baire category theorem, the set $\mathscr{E} = \cap_j \mathscr{E}_{a_j}$ is dense in $\mathscr{L}_{X}$. Now whenever $(f,E) \in \mathscr{E}$, then $E \cap \{a_j\}_j = \emptyset$ by construction, hence $E$ cannot have any interior point. This complete the proof.

\end{proof}  

We thus obtain \thref{THM:widebadsupp} and \thref{THM:widebadsuppw} as immediate consequences. 

\subsection{Sparse Fourier supports}
A non-empty subset $\Lambda \subset \T$ is said to be a \emph{Sidon set} if there exists $C(\Lambda)>0$ such that 
\[
\sum_n \abs{\widehat{T}(n)} \leq C(\Lambda) \sup_{\zeta \in \T} \abs{T(\zeta)}, 
\]
for any trigonometric polynomial with $\supp{\widehat{T}}\subseteq \Lambda$. Thus, Sidon sets are those subsets of integers $\Lambda$ for which all bounded functions with Fourier support in $\Lambda$ have $\ell^1$-summable Fourier coefficients. Consequently, the essentially bounded functions in \thref{THM:widebadsupp} and \thref{THM:widebadsuppw} cannot have Fourier support in a Sidon set, otherwise, their support would have interior points. The remarkable arithmetic characterization of Sidon sets due to G. Pisier (for instance, see \cite{pisier1983arithmetic} and also J. Borgain in \cite{bourgain1985sidon}) implies that the Fourier supports of classical Riesz products and lacunary Hadamard series are Sidon sets (see subsection 4.2). Therefore, methods based on sparse Fourier support are not well-suited for proving \thref{THM:widebadsupp} and \thref{THM:widebadsuppw}. However, as we will see in the next section, these techniques are quite powerful for constructing elements with pathological behavior wrt to range.




\section{Functions with exceptionally bad range} \label{SEC:LSBR}

\subsection{Simple summation lemmas}
Our main intention is to prove a Baire category version of \thref{THM:BADRANGE}. We start of with an elementary lemma on summation of positive sequences, that will be used repeatedly in our work. 

\begin{lemma}\thlabel{LEM:PhiwSum} Let $\Phi$ continuous function on $[0,1]$ with $\Phi(t)/t \downarrow 0$ as $t\to 0$. Then there exists positive numbers $(a_j)_j$ such that 
\[
\sum_j a_j = \infty, \qquad \sum_j a^2_j < \infty, \qquad \sum_j \Phi(a_j)< \infty.
\]
\end{lemma}
\begin{proof}
Without loss of generality, we may assume that $\sup_{0<t<1} \Phi(t)/t^2 =+\infty$, otherwise the task becomes much simpler. We shall use a simple argument, similar to \cite{limani2024fourier}. By the assumption $\Phi(t)/t \downarrow 0$ as $t\downarrow 0$, we can pick a (unique) sequence of positive integers $(t_n)_n$ which converge to zero and satisfy $\Phi(t_n)/t_n = 1/n^2$, for $n=1,2,3,\dots$. Now the idea is for each $n$, consider a uniform partition of the interval $[t_{n+1},t_n)$, of the form
\[
t_{n,k} = t_{n+1} + \frac{k}{N_n}(t_n-t_{n+1}), \qquad k=0,1,2,\dots, N_n -1,
\]
where $N_n>1$ positive integer to be determined. The idea is to uniformly pack in each interval $[t_{n+1},t_n)$, enough points that the sum of the points gets overwhelmingly big, while the $\Phi$-sum is finite. Note that we clearly have 
\[
\sum_{k=0}^{N_n -1} t_{n,k} \asymp N_{n}t_{n}, \qquad n=1,2,3,\dots.
\]
Now pick $N_n \asymp 1/t_n$, and denote the joint sequence $(t_{n,k})_{k,j}$ by $(w_n)_n$. On one hand, we have that 
\[
\sum_n w_n = \sum_{n=1}^\infty \sum_{k=1}^{N_n -1} t_{n,k} \asymp \sum_n 1 = \infty,
\]
Meanwhile, since $\Phi(t)/t$ is non-decreasing, we have 
\[
\sum_n \Phi(w_n) = \sum_{n=1}^\infty \sum_{k=1}^{N_n -1} \Phi(t_{n,k}) \leq \sum_{n=1}^\infty \frac{1}{n^2} \sum_{k=1}^{N_n -1} t_{n,k} \lesssim \sum_{n=1}^\infty \frac{1}{n^2} < \infty.
\]

\end{proof}
\subsection{Riesz products}
Here we gather some preliminary results on Riesz products, which will not only be useful in the remaining part of this section, but also in the next thereafter. For a detailed treatment of these matters, we refer the reader to Ch. V in the classical book of A. Zygmund in \cite{zygmundtrigseries}, or Ch. V in \cite{katznelson2004introduction} by Y. Katznelson.

Let $(a_j)_j$ be a sequence of real numbers with $-1<a_j <1$ and let $\{N_j\}_j$ be positive integers satisfying the condition $\kappa := \inf_j N_{j+1}/N_j\geq 3$.  We consider the associated Riesz partial products defined as 
\[
d\sigma_n(\zeta) = \prod_{j=1}^n \left(1 + a_j \Re(\zeta^{N_j}) \right) dm(\zeta), \qquad n=1,2,\dots \qquad \zeta \in \T.
\]
Using the assumption that $-1<a_j<1$, it is not difficult to show that $\{\sigma_n\}_n$ converges to a positive finite Borel measure $\sigma$ on $\T$, in the weak-star topology of complex finite Borel measure on $\T$. Sometimes, it will be convenient for us to refer to this weak-star star limiting measures $\sigma$ as \emph{the Riesz product}, associated with the parameters $(a_j)_j$ and $(N_j)_j$. We now record two important observations that will be useful for our purposes. First, a classical result by Zygmund (see Theorem 7.6, Ch. V in \cite{zygmundtrigseries}) asserts that $\sigma$ is mutually singular wrt $dm$, if and only if the corresponding sequence $(a_j)_j$ satisfies
\[
\sum_j a^2_j = \infty.
\]
However, the square divergent condition on $(a_j)_j$ alone does not exclude the support of $\sigma$ to be large (or even full). But if we additionally also impose that $a_j \to 0$, then $\sigma$ is actually supported on a set of Lebesgue measure zero. See Theorem 7.7, Ch. V in \cite{zygmundtrigseries}. Secondly, the assumption $N_{j+1}\geq 3N_j$ for each $j$ implies that we can expand
\[
Q_{k+1}(\zeta) := \prod_{j=1}^{k+1}\left( 1+a_j \Re(\zeta^{N_j} ) \right) = Q_k(\zeta) + a_{k+1} \Re \left( \zeta^{N_{k+1}}  Q_k(\zeta) \right), \qquad \zeta \in \T,
\]
where the last two terms in the expansion have disjoint Fourier support, that is, 
\[
\supp{\widehat{Q_k}} \cap \supp{\widehat{\zeta^{N_{k+1}} Q_k}} = \emptyset, \qquad k=1,2,3,\dots
\]
In fact, the Fourier frequencies of the term $Q_k$ are contained in the set $\{|n| \leq \sum_{j=1}^k N_j \}$, while the Fourier frequencies of the term $\zeta^{N_{k+1}}Q_k$ are supported inside the block
\[
\Lambda_{k+1}:= \left\{ n\in \mathbb{Z}: (1-(\kappa-1)^{-1}N_{k+1} \leq \abs{n} \leq  (1+(\kappa-1)^{-1})N_{k+1} \right\}, \qquad k=1,2,\dots.
\]
With this property at hand, we shall draw two important conclusions. First, we note that whenever $\Phi$ is a Young function, utilizing the property of disjoint Fourier support repeatedly, we obtain
 \[
\sum_n \Phi \left( \abs{\widehat{\sigma_{k+1}}(n)}\right) = \left(1+ \Phi (\abs{a_{k+1}}) \right)\sum_n \Phi \left( \abs{\widehat{\sigma_{k}}(n)}\right) = \prod_{j=1}^{k+1} \left(1+ \Phi (\abs{a_j}) \right), \qquad k=1,2,3,\dots.
\]
The same property also holds for limiting measure $\sigma$, hence we see that the Fourier coefficients of $\sigma$ belong to $\ell^{\Phi}$ if and only if $(a_n)_n$ is $\Phi$-summable. Secondly, for each $k=1,2,3,\dots$ 
\[
\sum_{n \in \Lambda_{k+1} } \Phi \left( \abs{\widehat{\sigma}(n)} \right) = \sum_{n\in \Lambda_{k+1}} \Phi \left( \abs{ \widehat{\zeta^{N_{k+1}}Q_k}(n) } \right) = \Phi(|a_{k+1}|) \prod_{j=1}^k \left(1+\Phi(|a_j|)\right).
\]

Moving forward, we now formulate the principle result in this subsection.

\begin{prop}\thlabel{PROP:MAINPHI} Let $\Phi$ be a Young function with $\Phi(t)/t \to 0$ as $t\to 0$. Then there exists a positive function $f \in L^2(\T,dm)$ with the following properties:
\begin{enumerate}
    \item[(i)] $\supp{fdm} = \T$,
    \item[(ii)] The Fourier coefficients of $f$ satisfy
    \[
    \sum_{n} \Phi \left( \abs{\widehat{f}(n)} \right) < \infty,
    \]
    \item[(iii)] $f$ is unbounded in a neighborhood of $\zeta=1$.
\end{enumerate} 
\end{prop}
\begin{proof}
Invoking \thref{LEM:PhiwSum}, we can choose positive numbers $(a_j)_j$ with with $0<a_j<1$, that additionally satisfy
\[
\sum_j a_j^2 < \infty, \qquad \sum_j \Phi \left( a_j \right) <\infty, \qquad \sum_j a_j = \infty.
\]
Consider the associated Riesz partial products
\[
d\sigma_N(\zeta) = \prod_{j=1}^N \left( 1+ a_j \Re(\zeta^{3^j}) \right) dm(\zeta), \qquad N=1,2,3,\dots \qquad \zeta \in \T,
\]
and let $\sigma$ be the Riesz products wrt these parameters. According to a classical Theorem of Zygmund (see Remark to Theorem 7.6, Ch. V in \cite{zygmundtrigseries}) the first condition ensures that $d\sigma = f dm$, where $f$ is a positive $L^2(\T, dm)$-function with $\supp{fdm}= \T$. The second condition guarantees that the Fourier coefficients of $f$ are $\Phi$-summable. At last, the divergence of $\sum_j a_j$ ensures that $f$ is unbounded in a neighborhood of $z=1$. 
\end{proof}
We remark that so far, we have only utilized that assumption that $\Phi$ is continuous with $\Phi(t)/t \downarrow 0$ as $t\downarrow 0$. Then Young function assumption will enter the picture in the next subsection, where define the corresponding Orlicz space. Our corresponding observation in the context of $\ell^1(w)$ goes as follows. 

\begin{prop}\thlabel{PROP:MAINl1w} Let $w_n \downarrow 0$. Then there exists a positive function $f\in L^2(\T,dm)$ with the following properties:
\begin{enumerate}
    \item[(i)] $\supp{fdm}= \T$,
    \item[(ii)] $f$ is unbounded in a neighborhood of $\zeta=1$,
    \item[(iii)] 
    \[
    \sum_n \abs{\widehat{f}(n)} w_{|n|}<\infty.
    \]
\end{enumerate}
\end{prop}

\begin{proof}
Pick sequence of positive integers $(N_j)_j$ which satisfies the properties
\[
\sum_j w_{N_j} < \infty, \qquad \inf_j \frac{N_{j+1}}{N_j} \geq 3.
\]
Now consider the sequence $a_j= 1/j$ for $j=1,2,3,\dots$ and recall that the following properties hold:
\begin{enumerate}
    \item[(i)] $\sum_j a^2_j < +\infty$,
    \item[(ii)] $\sum_j a_j = + \infty$,
    \item[(iii)] $a_{k+1} \prod_{j=1}^{k}(1+a_j) \leq 1$, for $k=1,2,3,\dots$
\end{enumerate}
Now consider the associated Riesz product defined by 
\[
d\sigma_M(\zeta) = \prod_{k=1}^M \left( 1+ a_j \Re(\zeta^{N_j} ) \right) dm(\zeta), \qquad \zeta \in \T, \qquad M=1,2,3,\dots.
\]
According to a classical Theorem of Zygmund (see remark to Theorem 7.6, Ch. V in \cite{zygmundtrigseries}) the condition $(i)$ ensures that $d\sigma= f dm$ such that $f\in L^2(\T,dm)$ positive with $\supp{fdm}= \T$. On the other hand, the condition $(ii)$ implies that $f$ is unbounded in a neighborhood of $\zeta=1$. Furthermore, using the disjoint Fourier support of Riesz products and $(iii)$, we get
\[
\sum_{n\in \Lambda_{j+1}} \abs{\widehat{f}(n) } = a_{j+1} \prod_{k=1}^j (1+a_k)\leq 1, \qquad j=1,2,3,\dots.
\]
With this observation at hand, and the monotonicity of $w_n$, we conclude that 
\[
\sum_{n} \abs{\widehat{f}(n)} w_{|n|} = \sum_{j=0}^\infty \sum_{
n\in \Lambda_{j+1}} \abs{\widehat{f}(n)} w_{|n|} \leq \sum_{j=0}^\infty w_{N_j} < \infty.
\]
\end{proof}
We remark the assumption similar to $(iii)$ in the proof above:
\[
\sup_{k\geq 0} \, a_{k+1} \prod_{j=1}^k (1+a_j) < \infty,
\]
actually implies that $\sum_k a^2_k < \infty$. This observation appeared in the work of P. Duren, see \cite{duren1965smoothness}, and implies that one cannot construct smooth Zygmund measures which are singular wrt $dm$, using classical Riesz products. 

\subsection{Another Baire category argument}
Here, we shall utilize a different Baire category argument in order to produce functions $f$ which are unbounded in every neighborhood in $\T$. Let $X$ denote either the Orlicz space $\ell^{\Phi}$ associated with a Young function or $\ell^1(w)$ with $w_n \downarrow 0$.

We denote by $X_+$ the cone of non-negative $L^2(\T,dm)$-functions whose Fourier coefficients belong to $X$, and equip this set with the metric
\[
\rho_X(f,g) = \norm{f-g}_{L^2} + \norm{f-g}_X.
\]
This gives rise to a complete separable metric space $(X_+, \rho_X)$. Indeed, in the case $X= \ell^{\Phi}$ this follows from the Young function assumption $\Phi(x/2)\asymp \Phi(x)$ as $x\to 0$. We now phrase our principal observation in this section, from which \thref{THM:BADRANGE} easily follows.

\begin{thm}\thlabel{THM:GENRANGE} Let $X$ be either equal to $\ell^{\Phi}$ with $\Phi$ Young function satisfying $\Phi(t)/t \downarrow 0$ as $t\downarrow 0$, or equal to $\ell^1(w)$ with $w_n \downarrow 0$. Then the corresponding set $X_+$ of functions $f$ which are bounded in some neighborhood of $\T$ is of first category in the metric space $(X_+, \rho_X)$.
\end{thm}

\begin{proof}[Proof of \thref{THM:GENRANGE}]

\proofpart{1}{The Baire argument:}
Let $a\in \T$, and fix numbers $\delta, M>0$. We consider sets of the form
\[
\mathscr{L}(a,\delta, M) = \left\{ f\in X_+:  m\left( \{\zeta \in I_{\delta}(a): f(\zeta) \geq M \}\right) =0\right\},
\]
where $I_{\delta}(a)$ denotes the arc centered at $a\in \T$ of length $\delta>0$. We shall prove that each set $\mathscr{L}(a, \delta, M)$ is closed and nowhere dense in the complete metric space $(X_+, \rho_X)$. 
Assuming for a moment that this is true, we may then pick a countable dense subset $A \subset \T$, a decreasing sequence of positive numbers $(\delta_k)_k$ with $\delta_k \downarrow 0$, and a sequence of positive numbers $(M_k)_k$ with $M_k\uparrow + \infty$. With this at hand, we observe that the set
\begin{multline*}
\bigcup_{a_k \in \T} \bigcup_{\delta_k >0} \bigcup_{M_k>0} \mathscr{L}(a_k, \delta_k, M_k) =  \\\left\{ f\in X_+: f  \, \text{bounded in some neighborhood of a point} \, a\in A \right\}.
\end{multline*}
must then be of the first category in the complete metric space $(X_+, \rho_X)$. Appealing to Baire's category Theorem, we conclude that the subset of $X_+$ consisting of functions which are unbounded in every neighborhood of $\T$, are in fact generic in $X_+$. This shows that the Theorem is proved, once the sets $\mathscr{L}(a,\delta, M)$ are shown to be closed and nowhere dense. We now turn to this precise task.
\proofpart{2}{Closed:} 
This part is simple, and only utilizes that elements in $X_+$ are integrable in $L^2(\T,dm)$. Assume $(f_k)_k \subset \mathscr{L}(a, \delta, M)$ which converges to $f\in X_+$. Fix an arbitrary $\eta>0$, and observe that for each $k$, we have
\begin{multline*}
m \left( \{ \zeta\in I_{\delta}(a): f(\zeta) > M+ \eta \} \right) \leq m \left( \{|f-f_k| > \eta  \} \right) + m \left( \zeta\in I_{\delta}(a): f_k(\zeta) >M \} \right) \\
\leq \frac{1}{\eta^2} \int_{\T} \abs{f-f_k}^2 dm \to 0, \qquad k\to \infty.
\end{multline*}
Letting $\eta \downarrow 0$ and using monotonicity, we conclude that $f\in \mathscr{L}(a,\delta,M)$.

\proofpart{3}{Nowhere dense:} It suffices to show that for any $g\in X_+$ and any $\varepsilon_0>0$, there exists $g_{0} \in X_+ \setminus \mathscr{L}(a, \delta, M)$, such that 
\[
\rho_X(g,g_0) < \varepsilon_0.
\]
By means of taking convolution with a positive smooth kernel, one can show that functions in $X_+$ which are continuous and zero-free on $\T$ are dense. Here, we intrinsically make use of the assumption that $X$ is separable. Without loss of generality, we may therefore assume that $g$ is continuous on $\T$. According to \thref{PROP:MAINPHI} and/or \thref{PROP:MAINl1w}, we can exhibit a element $f \in X_+$, with $\norm{f}_{X}=1$, which is unbounded in a neighborhood of $a\in \T$. It now readily follows that the elements 
\[
g_{\varepsilon,a}(\zeta)= g(\zeta) + \varepsilon \cdot f(\zeta \conj{a}), \qquad \varepsilon>0, \qquad \zeta \in \T
\]
belong to $X_+$, and are unbounded in neighborhoods of the point $a\in \T$. It follows that 
\[
\rho_X(g_{\varepsilon,a}, g) = \varepsilon \norm{f}_{L^2} + \varepsilon < \varepsilon_0,
\]
provided that $\varepsilon>0$ is sufficiently small. We therefore conclude that $X_+ \setminus \mathscr{L}(a, \delta, M)$ is dense in $X_+$. The proof is now complete.
\end{proof}

\subsection{Lacunary Hadamard series and random Fourier series}
Using lacunary series, we briefly include some explicit constructions of a real-valued functions $f\in L^2(\T,dm)$, which are unbounded in every neighborhood of $\T$, and whose Fourier coefficients belong to either $\ell^{\Phi}$ or $\ell^1(w)$. Now, given real-numbers $(a_j)_j$, and a lacunary sequence of positive integers $(N_j)_j$, with $\inf_j N_{j+1}/N_j \geq 2$, we consider the corresponding Lacunary Hadamard series 
\begin{equation}\label{EQ:LACFOUR}
F(\zeta) := \sum_{j=0}^{\infty} a_j \zeta^{\pm N_j}, \qquad \zeta \in \T.
\end{equation}
Clearly, $F \in L^2(\T, dm)$ if and only if $\sum_j \abs{a_j}^2 < \infty$. In fact, the lacunary assumption of $\supp{\widehat{f}}=\{N_j\}_j$ actually implies that the associated Littlewood-Paley Square function of $F$ is uniformly bounded on $\T$:
\[
\mathcal{S}(F)(\zeta) = \left(\sum_{j=0}^\infty \, \abs{ \sum_{2^j \leq |n| <2^{j+1}} \widehat{F}(n) \zeta^n }^2 \right)^{1/2} \leq \left( \sum_{j=0}^\infty \abs{a_j}^2 \right)^{1/2} < \infty, \qquad \zeta \in \T.
\]
This proves that $F$ belongs to $BMO(\T)$:
\[
\sup_I \frac{1}{m(I)} \int_I \abs{F(\zeta)- F_I} dm(\zeta) < \infty,
\]
where the supremum is taken over all arcs $I\subseteq \T$. See Ch. IV in \cite{stein1993harmonic}. Furthermore, a classical result of A. Zygmund (see Lemma 6.11, Ch VI in \cite{zygmundtrigseries}) asserts that if $\sum_j \abs{a_j} = \infty$, then $F$ is unbounded in everywhere neighborhood of $\T$. We now outline the main arguments:

Now if $\Phi(t)/t \downarrow 0$ as $t \downarrow 0$, \thref{LEM:PhiwSum} allows us to construct positive numbers $(a_j)_j$ with the properties
\[
\sum_j a_j = \infty, \qquad \sum_j a^2_j <\infty, \qquad \sum_j \Phi(a_j)< \infty.
\]
It readily follows from the above observations that the corresponding lacunary Hadamard series $F$ in \eqref{EQ:LACFOUR} satisfies the following properties:
\begin{enumerate}
    \item[(i)] $F$ belongs to $BMO(\T)$.
    \item[(ii)] $F$ is unbounded in every neighborhood of $\T$.
    \item[(iii)] $(\widehat{F}(n))_n$ belongs to $\ell^{\Phi}$.
\end{enumerate}
In the context of $\ell^1(w)$, we instead appeal to to \thref{LEM:L2notl1w} below, which allows us to exhibit positive numbers $(a_j)_j$ with the properties 
\[
\sum_j a_j = \infty,\qquad  \sum_j a^2_j < \infty, \qquad \sum_j a_j w_{|j|}< \infty.
\]
The corresponding lacunary Hadamard series $F$ then satisfies the above properties $(i)-(ii)$, but now the Fourier coefficients of $F$ instead belong to $\ell^1(w)$. On the passing, we also mention that constructions of the same kind can be generated using random Fourier series of the kind
\[
\sum_n \varepsilon_n a_n \zeta^n, \qquad \zeta \in \T,
\]
where $(\varepsilon_n)_n$ are random signs with $\varepsilon_n \in \{-1,1\}$ for each $n$. The principal tool involved in this context is then a Khinchine-type inequality. For instance, see Ch. 5 in \cite{stein2011functional} for further details.

\section{Summable elements with Small supports}\label{SEC:SMALLSUPP}
\subsection{Measures with small support in $\ell^\Phi$}

Here we shall give a simple and completely different proof of \thref{THM:IMPhi} on a measure supported on a small set with $\Phi$-summable Fourier coefficients, as initially announced in \cite{limani2024fourier}. Below, our proof will instead utilize Riesz products. 

\begin{proof}[Proof of \thref{THM:IMPhi}] 
Since $\Phi(t)/t^2 \downarrow 0$ as $t\downarrow 0$, an argument very similar to \thref{LEM:PhiwSum} allows for the construction of positive numbers $(a_j)_j$ in the interval $(0,1)$ with the properties
\[
\sum_{j} a^2_j = \infty, \qquad \sum_j \Phi(a_j) < \infty.
\]
We now claim that the corresponding Riesz product
\[
d\sigma(\zeta) := \lim_{N\to \infty} \prod_{j=1}^N \left( 1+ a_j \Re (\zeta^{3^j}) \right)dm(\zeta), \qquad \zeta \in \T,
\]
interpreted in the sense of weak-star convergence of measures in $\T$, satisfies the required properties. Indeed, since $a_j \to 0$ and $(a_j)_j$ square divergent, Theorem 7.7 from Ch. V in \cite{zygmundtrigseries} ensures that the probability measure $\sigma$ is supported on a set of Lebesgue measure zero. Meanwhile, the $\Phi$-summability of $(a_j)_j$ in conjunction with observations from subsection 4.2 implies that the Fourier coefficients of $\sigma$ are $\Phi$-summable.
\end{proof}

\subsection{Measures with small support in $\ell^1(w)$}
In the context weighted $\ell^1$, we actually prove a slightly more general statement.

\begin{thm}\thlabel{THM:singl11w} Let $(w_n)_n$ positive numbers with $\inf_n w_n =0$. Then there exists a probability measure $\sigma$ supported on a set of Lebesgue measure zero, such that 
\[
\sum_n \abs{\widehat{\sigma}(n)} w_{|n|} < \infty.
\]
\end{thm}
It is remarkable to note that this result is as sharp as possible, since if $\inf_n w_n >0$, then $\ell^1(w)$ is contained in the classical Wiener algebra $\ell^1$ which only consists of continuous functions on $\T$.
\begin{proof} Since the positive numbers $(w_n)_n$ satisfy $\inf_n w_n = 0$, we can find a subsequence of positive integers $(N_j)_j$, such that 
\[
\sum_j w_{N_j}2^j < \infty.
\]
By means of passing to yet another subsequence, and thus further enlarging the gaps of $N_j$, we may also assume that $\inf_{j} N_{j+1}/N_j \geq 3$. Pick a sequence of positive numbers $(a_j)_j$ with $0<a_j <1$ which satisfy the properties:
\[
\sum_j a^2_j = +\infty, \qquad a_j \to 0.
\]
Now consider the Riesz product associated to these parameters $(N_j)_j$ and $(a_j)_j$, defined as
\[
d\sigma(\zeta) = \lim_{N\to \infty} \prod_{j=1}^N \left(1+a_j \Re(\zeta^{N_j}) \right)dm(\zeta), \qquad \zeta \in \T,
\]
interpreted in the sense of weak-star convergence of measures. The assumptions on $(a_j)_j$ and Zygmund's Theorem once again ensure that $\sigma$ is probability measure supported on a set of Lebesgue measure zero. Meanwhile, the lacunary assumption on $(N_j)_j$ implies
\[
\sum_n \abs{\widehat{\sigma}(n)}w_{|n|} \leq \sum_{j=0}^\infty w_{N_j} \sum_{n\in \Lambda_{j+1}} \abs{\widehat{\sigma}(n)} \leq \sum_{j=0}^\infty w_{N_j} a_{j+1} \prod_{k=1}^j (1+a_k) \leq \sum_{j\geq 0} w_{N_j} 2^{j} <\infty.
\]
This complete the proof.
\end{proof}

\section{Thresholds of Katznelson-type} \label{SEC:THREKATZ}

\subsection{Continuous functions outside of $\ell^1(w)$}
The following proposition will be our principal result in order to establish \thref{THM:ContNOTl1w}, and may be interesting in its own right.

\begin{prop}\thlabel{PROP:l1w} Let $(w_n)_n$ be positive numbers satisfying the following properties:
\begin{enumerate}
    \item[(i)] $\sum_{n} w^2_n = +\infty$,
    \item[(ii)] $w_n$ is non-increasing, or satisfies the doubling condition: there exists $C>1$ such that for all $n\geq1$:
    \[
    C^{-1} w_n \leq w_k \leq C w_n, \qquad n\leq k \leq 2n.
    \] 
\end{enumerate}  
Then for any $0<\gamma <1/10$ and any $N>0$, there exists real-valued trigonometric polynomials $T_{N,\gamma}$, such that 
\[
\sup_{\zeta \in \T} \abs{T_{N,\gamma}(\zeta)}\leq 10\gamma, \qquad \supp{\widehat{T_{N,\gamma}}} \subseteq \{|n|\geq N\}, \qquad \sum_n \abs{\widehat{T_{N,\gamma}}(n)}w_n \geq \frac{1}{10}\gamma^{-1/3}
\]
\end{prop}

With this at hand, we readily deduce our principle result. 

\begin{proof}[Proof of \thref{THM:ContNOTl1w}] Let $(\gamma_j)_j$ with $\sum_j \gamma_j < \infty$. According to \thref{PROP:l1w}, we can find a trigonometric polynomial $T_1 := T_{1,\gamma_1}$ such that 
\[
\sup_{\zeta \in \T} \abs{T_1(\zeta)}\leq 10\gamma_1, \qquad \supp{\widehat{T_1} } \subseteq \{1\leq |n|\leq N_1(\gamma_1)\}, \qquad \sum_n \abs{\widehat{T_{1}}(n)}w_n \geq \frac{1}{10}\gamma_1^{-1/3}
\]
Next, pick $T_2 := T_{2N_1(\gamma_1), \gamma_2}$, such that 
\[
\sup_{\zeta \in \T} \abs{T_{2}(\zeta)}\leq 10\gamma_2, \qquad \supp{\widehat{T_{2}}} \subseteq \{2N_1(\gamma_1)\leq |n| \leq N_2(\gamma_2)\}, \qquad \sum_n \abs{\widehat{T_{2}}(n)}w_n \geq \frac{1}{10}\gamma_2^{-1/3}.
\]
Iterating this argument and making use of \thref{PROP:l1w} in this manner, we obtain a sequence of trigonometric polynomials $(T_j)_j$ with mutually disjoint Fourier support that satisfy 
\[
\sup_{\zeta \in \T} \abs{T_j(\zeta)}\leq \gamma_j, \qquad \sum_n \abs{\widehat{T_{j}}(n)}w_n \geq \frac{1}{10}\gamma_j^{-1/3}, \qquad j=1,2,3,\dots
\]
Now set $f= \sum_j T_j$ and note that the assumption $\sum_j \gamma_j < \infty$ in conjunction with Weierstrass $M$-test ensures that $f$ is continuous on $\T$. On the other hand, since the $(T_j)_j's$ have mutually disjoint Fourier support, we also have 
\[
\sum_{n>0} \abs{\widehat{f}(n)} w_{n} = \sum_j \sum_n \abs{\widehat{T_{j}}(n)}w_n \geq \frac{1}{10} \sum_j \gamma_j^{-1/3} = + \infty.
\]
This completes the proof of our Theorem.
\end{proof}

The rest of this subsections will thus be devoted to proving \thref{PROP:l1w}, which will require two lemmas. The first lemma  involves the construction of an $L^2$-function whose $\ell^1(w)$ partial sums are relative larger than a power of the partial sums of $w^2_j$.
\begin{lemma}\thlabel{LEM:L2notl1w} Let $(w_n)_n$ be positive numbers satisfying the following properties:
\begin{enumerate}
    \item[(i)] $\sum_{n} w^2_n = +\infty$,
    \item[(ii)] $w_n$ is non-increasing, or satisfies the doubling condition: there exists $C>1$ such that for all $n\geq1$:
\end{enumerate}  
Then there exists positive numbers $(a_n)_n$ such that 
\begin{enumerate}
    \item[(a.)] $\sum_n a^2_n < \infty$,
    \item[(b.)] $\sum_n a_n w_{n} =+\infty$,
   \item[(c.)] For any positive integers $N>M>0$, we have
    \[
  \sum_{M\leq n < N} a_n w_n  \geq  \frac{1}{4} \left(1+ \sum_{M\leq n < N} w^2_n \right)^{1/3}.
   \]
\end{enumerate}

\end{lemma}
We remark that exponent $1/3$ in the statement of $(c.)$ is chosen for convenience, and can easily be modified to any desirable exponent $0<\eta<1/2$, as we shall see below.
\begin{proof}
Note that under any of the assumptions in $(ii)$, it easily follows by the Cauchy condensation test that 
\[
\sum_n 2^n w^2_{2^n} = +\infty.
\]
For simplicity, we only outline the proof when $(w_n)_n$ is non-increasing, as the proof can easily be adapted by only using the doubling-type assumption. To this end, we shall take the intended sequence to be 
\[
a_k = \frac{w_{2^n}}{\left(1+ \sum_{j=1}^{n-1} 2^j w^2_{2^j} \right)^{2/3}}, \qquad k\in [2^{n-1}, 2^{n} ), \qquad  n=1,2,3,\dots
\]
Now, consider the continuous non-increasing function on $[1,\infty)$, whose graph is a piecewise linear interpolation of the integer-values $s(n)=2^n w^2_{2^n}$. Now it is not difficult to see using the integral test that 
\[
\sum_n a^2_n \asymp \sum_n \frac{2^n w^2_{2^n}}{\left(1+ \sum_{j=1}^{n-1} 2^j w^2_{2^j} \right)^{4/3}} \asymp \int_{w_1}^\infty \frac{s(x)}{\left(1+\int_1^x s(t) dt\right)^{4/3}} dx = \int_{1}^\infty \frac{dy}{(1+y)^{4/3}} <\infty.
\]
In the last step, we utilized the change of variable $y=S(x)$, where $S(x):= \int_1^x s(t)dt$ is continuously differentiable and strictly increasing with $S(x)\uparrow + \infty$ as $x\to \infty$. This proves that $(a_n)_n$ is square summable, hence $(a.)$ holds. Now in order to prove $(b.)$ it suffices to prove $(c.)$ and appeal to the divergence of $\sum_j w^2_j$. To this end, utilizing the monotonicity of $(w_n)_n$, we have 
\[
\sum_{2^{j-1} \leq n < 2^j} a_{n} w_{n} \geq \frac{1}{2} \frac{2^j w^2_{2^j}}{\left(1+ \sum_{k=1}^{j-1} 2^k w^2_{2^k} \right)^{2/3} }, \qquad j=1,2,3,\dots
\]
With this observation at hand, we again appeal to the integral test, and carry out the same change of variable as before. We then deduce the following:
\begin{multline*}
\sum_{2^M \leq n < 2^N} a_{n} w_{n} \geq \frac{1}{2} \sum_{j=M+1}^{N}  \frac{2^j w^2_{2^j}}{\left(1+ \sum_{k=1}^{j-1} 2^k w^2_{2^k} \right)^{2/3} } \geq \frac{1}{2} \int_{M+1}^N  \frac{s(x)}{\left(1+\int_1^x s(t) dt\right)^{2/3}} dx = \\ \frac{3}{2} \left( 1+ \int_1^N s(t)dt\right)^{1/3} - \frac{3}{2}\left(1+ \int_1^{M+1} s(t)dt \right)^{1/3} \geq \frac{3}{4} \left( 1+ \int_{M+1}^N s(t)dt \right)^{1/3} \\ \geq \frac{3}{4}\left( 1+ \sum_{2^{M+1}\leq n< 2^N} w^2_n \right)^{1/3}.
\end{multline*}
In the penultimate step, we utilized the elementary inequality 
\[
x^{1/3}+y^{1/3} \leq 2(x+y)^{1/3}, \qquad x,y>0,
\]
while in the last step, we also utilized the monotonicity of $(w_n)_n$. It is clear that a similar argument also works when $M,N$ are not necessarily powers of $2$, we omit the details.

\end{proof}

Our second lemma is classical probabilistic result on sums of Rademacher Random variables, whose proof is simple and outlined on p. 277 in Appendix B of \cite{katznelson2004introduction}.

\begin{lemma}\thlabel{LEM:RADKATZ} For any sequence of positive numbers $(a_n)_n$ with $a^2 := \sum_n a^2_n< \infty$ and any $\lambda>0$, there exists a sequence $(\varepsilon_n)_n$ of signs $\varepsilon_n \in \{-1,1\}$ such that for the corresponding $L^2$-function
\[
f(\zeta) := \sum_n \varepsilon_n a_n \zeta^n, \qquad \zeta \in \T
\]
the following estimate holds:
\[
\norm{f-f^\lambda }^2_{L^2} \leq 4\left(\lambda^2 + 2a^2\right) \exp \left( - \frac{\lambda^2}{2a^2} \right), \qquad \lambda >0,
\]
where $f^{\lambda} := \min\left( \lambda , |f| \right) \frac{f}{|f|}$ is the truncation by height $\lambda$.
\end{lemma}

We are ready to prove our main proposition in this subsection.

\begin{proof}[Proof of \thref{PROP:l1w}] Note that by means of substituting $w_n$ with a constant multiple of $\min(1/\sqrt{n}, w_n)$, we may assume that $\sup_{n\geq 0}\sqrt{n}w_n \leq 1$. Fix a number $0<\gamma<1/10$, and positive integer $N>0$. According to \thref{LEM:L2notl1w}, we can find positive numbers $(a_n)_n$ with the properties:
\[
\sum_n a^2_n < +\infty, \qquad \sum_n a_n w_n = +\infty, \qquad \sum_{M\leq n < L} a_n w_n  \geq  \frac{1}{4} \left(1+ \sum_{M\leq n < L} w^2_n \right)^{1/3},
\]
for all $0<M<L$. Now pick a positive number $M(\gamma)\geq N$, such that 
\[
\sum_{n\geq M(\gamma)} a^2_n \leq \gamma^4, 
\]
Let $N(\gamma)>M(\gamma)$ be the largest positive integers such that
\[
\sum_{M(\gamma) \leq n < N(\gamma)} w^2_{n} \leq \gamma^{-2}
\]
This is possible since $w_{n} \to 0$ and $\sum_n w^2_{n} = + \infty$. Applying \thref{LEM:RADKATZ} to the truncated sequence $(a_n)_{n\geq M(\gamma)}$, we can find signs $(\varepsilon_j)_j$, such that the $L^2$-function
\[
f_\gamma (\zeta) = \sum_{n\geq M(\gamma)} \varepsilon_n a_n \zeta^n, \qquad \zeta \in \T,
\]
satisfies
\[
\norm{f_\gamma -(f_\gamma)^\lambda}^2_{L^2} \leq 4\left( \lambda^2 + 2\gamma^4 \right) \exp \left(- \frac{\lambda^2}{2\sum_{n\geq M(\gamma)} a^2_n} \right), \qquad \lambda >0.
\]
Choosing $\lambda_{\gamma} = (\sum_{n\geq M(\gamma)} a^2_n)^{1/2}$, we obtain 
\[
\norm{f_\gamma -(f_\gamma)^{\lambda_\gamma}}_{L^2} \leq 2\gamma^2.
\]
We shall deduce that the $\ell^1(w)$-norm of $(f_{\gamma})^{\lambda_\gamma}$, restricted to frequencies in $M(\gamma)\leq |n|< N(\gamma)$, is large enough. To this end, note that 
\[
\sum_{M(\gamma)\leq n < N(\gamma)} \abs{\widehat{(f_\gamma -(f_\gamma)^\lambda )}(n)} w_n \leq \left( \sum_{M(\gamma)\leq n <N(\gamma)} w^2_{n} \right)^{1/2} \cdot \norm{f_\gamma -(f_\gamma)^{\lambda_\gamma}}_{L^2} \leq 2 \gamma.
\]
With this observation at hand, in conjunction with the hypothesis $(iii)$ of $(a_n)_n$ in \thref{LEM:L2notl1w}, we obtain
\begin{multline*}
\sum_{M(\gamma)\leq n <N(\gamma)} \abs{\widehat{(f_\gamma)^{\lambda_\gamma} )}(n)} w_{n} \geq \sum_{M(\gamma)\leq n <N(\gamma)} \abs{\widehat{f_\gamma}(n)} w_n - 2\gamma 
= \sum_{M(\gamma)\leq n \leq N(\gamma)} a_n w_n - a_{N(\gamma)}w_{N(\gamma)} - 2\gamma \\ \geq \frac{1}{4} \left( 1 + \sum_{M(\gamma)\leq n \leq N(\gamma)} w^2_{n} \right)^{1/3} -3\gamma \geq \frac{1}{4}\left( 1 + \frac{1}{\gamma^2} \right)^{1/3} - 3\gamma \geq \frac{1}{10}\gamma^{-2/3}.
\end{multline*}
In the last step, we used an elementary calculus estimate, which holds whenever $0<\gamma<1/10$. Now let $\mathcal{F}_J$ denote the Fej\'er kernel of order $J$ and consider the difference of de la Vall\'ee Poussin kernels,
\[
K_{\gamma} := 2\mathcal{F}_{2N(\gamma)-1} - \mathcal{F}_{N(\gamma)-1} - 2\mathcal{F}_{2M(\gamma)-1} + \mathcal{F}_{M(\gamma)-1},
\]
which are trigonometric polynomials of degree at most $2N(\gamma)+1$, and with Fourier support 
\[
\widehat{K_\gamma}(n)= \begin{cases} 1 & M(\gamma)\leq |n|\leq N(\gamma), \\
0 & \, \text{otherwise}
    
\end{cases}
\]
Taking convolution of $(f_\gamma)^{\lambda_\gamma}$ with $K_\gamma$, we obtain the trigonometric polynomials
\[
T_{\gamma}(\zeta) := (f_\gamma)^{\lambda_\gamma} \ast K_\gamma (\zeta) =  \sum_{M(\gamma)\leq |n|\leq N(\gamma)} \widehat{(f_\gamma)^{\lambda_\gamma}}(n) \zeta^n, \qquad \zeta \in \T,
\]
which satisfy the following properties
\[
\sup_{\zeta \in \T} \abs{T_\gamma(\zeta)} \leq  \norm{K_{\gamma}}_{L^1} \lambda_\gamma \leq 6\gamma^2 , \qquad \sum_{M(\gamma)\leq |n| <N(\gamma)} \abs{\widehat{T_\gamma}(n)} w_{n} \geq \frac{1}{10}\gamma^{-2/3},
\]
and the Fourier support of $T_\gamma$ is contained in $[M(\gamma),N(\gamma)]\subseteq [N, \infty)$. The proof is now complete, modulo a trivial re-scaling of the parameter $\gamma \mapsto \sqrt{\gamma}$.
\end{proof}

\subsection{Continuous functions which are not $\Phi$-summable}
Here we outline the slight extensions of Katznelson's results, which both principally rely on the following classical lemma. For instance, see p. 125 in \cite{katznelson2004introduction}.

\begin{lemma}[Classical lemma] \thlabel{LEM:Cla} For any $\varepsilon>0$, there exists a trigonometric polynomials $T_{\varepsilon}$ in $\T$ such that
\[
\sup_{\zeta \in \T} \, \abs{T_{\varepsilon}(\zeta)} \leq 1, \qquad \norm{T_\varepsilon}_{L^2} \geq \frac{1}{2}, \qquad \sup_n \abs{\widehat{T_\varepsilon}(n)}\leq \varepsilon.
\]
\end{lemma}

With this lemma at hand, we turn to the proof.

\begin{proof}[Proof of \thref{THM:KATZ1}]
Let $\varepsilon>0$ and let $T_{\varepsilon}$ be as in \thref{LEM:Cla}, and note that by means of multiplying $T_{\varepsilon}$ with $\zeta^N(\varepsilon)$ for $N(\varepsilon)>0$ large enough, we may also assume that $\widehat{T}_{\varepsilon}(0)=0$. Fix positive real numbers $(A_j)_j$ with the property $\sum_j A_j <\infty$. Utilizing the assumption $\Psi(t)/t^2 \downarrow 0$ as $t\downarrow 0$, one can inductively construct positive numbers $(\varepsilon_j)_j$, such that
\[
\frac{\Psi(A_j \varepsilon_j)}{A^2_j\varepsilon^2_j} \geq \frac{1}{A^2_j}, \qquad j=1,2,3, \dots
\]
For this choice of $(\varepsilon_j)_j$, we again by means of induction, construct positive integers $(N_j)_j$ lacunary enough, so that $\zeta^{N_j}T_{\varepsilon_j}(\zeta)$ and $\zeta^{N_k}T_{\varepsilon_k}(\zeta)$ have have mutually disjoint Fourier support when $j\neq k$. With these choices of parameters at hand, we now consider the function of the form
\[
f(\zeta) = \sum_j A_j \zeta^{N_j} T_{\varepsilon_j}(\zeta), \qquad \zeta \in \T.
\]
Note that the summability condition on $(A_j)_j$ in conjunction with the Weierstrass $M$-test ensures that $f$ is continuous on $\T$. On the other hand, we also have that
\[
\sum_n \Psi \left( A_j \abs{\widehat{T_{\varepsilon_j}}(n)}\right) \geq \frac{\Psi(A_j \varepsilon_j)}{\varepsilon^2_j} \sum_{n} \abs{\widehat{T_{\varepsilon_j}}(n)}^2 \geq \frac{1}{4}, \qquad j=1,2,3,\dots
\]
With this observation at hand, it follows that
\[
\sum_n \Psi \left( \abs{\widehat{f}(n)} \right) = \sum_j \sum_n \Psi \left( A_j \abs{\widehat{T_{\varepsilon_j}}(n)} \right) \geq  \frac{1}{4} \sum_j 1 = + \infty.
\]
This completes the proof.
\end{proof}

\subsection{The dual problem on Fourier-Stieltjes coefficients}
Here we briefly demonstrate how \thref{THM:KATZ2} and \thref{COR:FOURSTIL} follow from \thref{THM:KATZ1} and \thref{THM:ContNOTl1w}, respectively. We start with the weighted setting, whose proof is a simple duality argument.

\begin{proof}[Proof of \thref{COR:FOURSTIL}]
For the sake of obtaining a contradiction, suppose that the Banach space $\ell^\infty(1/w)$ is continuously contained in the Banach space $M(\T)$ of complex finite Borel measures on $\T$, equipped with the usual total variation norm $\norm{\mu}_{M(\T)}$. By the closed graph Theorem, there exists $C_w>0$, such that
\[
\norm{\mu}_{M(\T)} \leq C_w \sup_{n} \frac{\abs{\widehat{\mu}(n)}}{w_{|n|}}.
\]
Fix an arbitrary $f\in C(\T)$, an integer $N>0$, and note that by duality between $\ell^1(w)$ and $\ell^\infty(1/w)$, we have
\begin{multline*}
\sum_{|n|\leq N} \abs{\widehat{f}(n)}w_{|n|} = \sup_{\substack{\norm{(a_n)_n}_{\ell^\infty(1/w)}\leq 1 \\ a_n=0 \, \, |n|>N }} \abs{\sum_{n} \widehat{f}(n) \conj{a_n} } \leq \sup_{\norm{\mu}_{M(\T)}\leq C_w}\abs{\sum_{n} \widehat{f}(n) \conj{\widehat{\mu}(n)} } =  \\ \sup_{\norm{\mu}_{M(\T)} \leq C_w}\abs{\int_{\T} f(\zeta) d\conj{\mu}(\zeta) } \leq C_w \sup_{\zeta \in \T} \abs{f(\zeta)}.
\end{multline*}
However, this contradicts \thref{THM:ContNOTl1w}, by means of letting $N\to \infty$.
\end{proof}

A similar proof of \thref{THM:KATZ2} can also be given, but shall below outline a different proof, which additionally justifies why it may be interpreted as a statement about elements with maximally bad range in measure theoretical sense.
\begin{proof}[Proof of \thref{THM:KATZ2}]
A simple adaption of \thref{LEM:PhiwSum} allows us to construct $(c_j)_j$ be a sequence of positive numbers with 
\[
\sum_j c^2_j = \infty, \qquad \sum_j \Phi(c_j) < \infty.
\]
Consider the distribution $S$ on $\T$ with Fourier coefficients defined by
\[
\widehat{S}(n) = \begin{cases} c_{|j|} & j= \pm 2^n \\
0 & j\neq \pm 2^n.
\end{cases}
\]
By assumption, it follows that
\[
\sum_{n} \Phi \left( \abs{\widehat{S}(n)} \right) < +\infty.
\]
Since $\sum_j c^2_j= \infty$ it follows from a classical Theorem of Zygmund on Lacunary series that the Cauchy integral of $S$:
\[
\Ka (S)(z) := \sum_{n=0}^\infty \widehat{S}(n) z^n = \sum_{j=0}^\infty c_j z^{2^j}, \qquad |z|<1,
\]
has radial limit at almost no point in $\T$ wrt $dm$. Now if there would exists a complex finite Borel measure $\mu$ on $\T$ with
\[
\widehat{\mu}(n) = \widehat{S}(n) , \qquad n\in \mathbb{Z},
\]
then the Cauchy integral of $\mu$:
\[
\Ka(\mu)(z) = \sum_{n=0}^\infty \widehat{\mu}(n) z^n = \Ka(S)(z), \qquad |z|<1 ,
\]
would have finite radial limit at almost every point, according to a classical Theorem of Smirnov. For instance, see Theorem 2.1.10 in \cite{cauchytransform}. The statement now follows with $a_n = \widehat{S}(n)$.
\end{proof}

Let $S$ be the distribution with $\widehat{S}(n)=a_n$ for all $n$, where $(a_n)_n$ is either as in \thref{COR:FOURSTIL} or as in \thref{THM:KATZ2}. Then the proof in previous paragraph in conjunction with Plessner's Theorem (cf. Theorem 2.5 in Ch. VI of \cite{garnett2005harmonic}), asserts that for $dm$-a.e $\zeta \in \T$, the image of $\Ka(S)$ under any Stolz-angle centered at $\zeta$, is dense in the complex plane $\C$.

\bibliographystyle{siam}
\bibliography{mybib}

\Addresses

\end{document}